\documentclass[11pt]{article}
\usepackage{geometry,graphicx,latexsym,amssymb,verbatim, multicol}
\usepackage{amsmath, amssymb, amsthm}
\usepackage{tikz}
\usepackage{epic}
\usepackage{pstricks}
\usepackage{pst-node}
\usepackage{pstricks-add}
\usepackage[numbers,sort&compress]{natbib}
\usetikzlibrary{calc}

\newtheorem{thm}{Theorem}[section]
\newtheorem{lem}[thm]{Lemma}

\newtheorem{prob}{Problem}
\newtheorem{cor}[thm]{Corollary}
\newtheorem{obs}[thm]{Observation}

\newtheorem{defi}{Definition}

\renewcommand{\comment}[1]{}

\renewenvironment{proof}{\noindent {\it Proof.}}{$\Box$\\}
\newcommand{\dontshow}[1]{}

\newcommand{\ov}{\overline}
\newcommand{\modulo}[1]{\hspace{1ex}(\hspace{-1.5ex} \mod {#1})}

\begin{document}

\begin{center}
{\LARGE  New approach to the $k$-independence number of a graph}
\mbox{}\\[8ex]

\begin{multicols}{2}

Yair Caro\\[1ex]
{\small Dept. of Mathematics and Physics\\
University of Haifa-Oranim\\
Tivon 36006, Israel\\
yacaro@kvgeva.org.il}

\columnbreak

Adriana Hansberg\\[1ex]
{\small Dep. de Matem\`atica Aplicada III\\
UPC Barcelona\\
08034 Barcelona, Spain\\
adriana.hansberg@upc.edu}\\[2ex]

\end{multicols}

\end{center}

\begin{abstract}
Let $G = (V,E)$ be a graph and $k \ge 0$ an integer. A $k$-independent set $S \subseteq V$ is a set of vertices such that the maximum degree in the graph induced by $S$ is at most $k$. With $\alpha_k(G)$ we denote the maximum cardinality of a $k$-independent set of $G$. We prove that, for a graph $G$ on $n$ vertices and average degree $d$, $\alpha_k(G) \ge \frac{k+1}{\lceil d \rceil + k + 1} n$, improving the hitherto best general lower bound due to Caro and Tuza [Improved lower bounds on k-independence, J. Graph Theory 15 (1991), 99-107].\\

\end{abstract}


\section{Introduction}

Let $G = (V,E)$ be a graph on $n$ vertices and $k \ge 0$ an integer. A \emph{$k$-independent set} $S \subseteq V$ is a set of vertices such that the maximum degree in the graph induced by $S$ is at most $k$. With $\alpha_k(G)$ we denote the maximum cardinality of a $k$-independent set of $G$ and it is called the \emph{$k$-independence number} of $G$. In particular, $\alpha_0(G) = \alpha(G)$ is the usual independence number of $G$. The Caro-Wei bound $\alpha(G) \ge \sum_{v \in V} \frac{1}{\deg(v)+1}$  \cite{Caro, Wei} is an improvement of the well-known Tur\'an bound for the independent number $\alpha(G) \ge \frac{n}{d(G)+1}$ \cite{Tur}, where $d(G)$ is the average degree of $G$. Various results concerning possible improvements and generalizations of the Caro-Wei bound are known (see \cite{AjKoSz1, AlKaSe, AlSp, BCHN, BGHR, GHRSch, Gri, HaRa, HaSch, Mur, Sel, Shea1, Shea2}). A well known generalization to the $k$-independence number of $r$-uniform hypergraphs was obtained by Caro and Tuza \cite{CaTu} improving earlier results of Favaron \cite{Fav} and was extended to non-uniform hypergraphs by Thiele \cite{Thi}. See also the recent papers \cite{CFHV_survey, CPS} for updates. An extension of the notion of residue of a graph, notably developed by Fajtlowicz in \cite{Faj} and Favaron et al. in \cite{FaMaSa}, to the notion of $k$-residue has been developed by Jelen \cite{Jel}. There has been also much interest in using the Caro-Tuza to algorithmic aspects (see \cite{HaLa, Los, ShaSri}). Yet all these lower bounds give asymptotically $\alpha_k(G)\ge \frac{k+2}{2(d+1)}n$ for $k$ fixed and $d = d(G)$. It is easy to see that in general we cannot hope to get better than $\frac{k+1}{d+1}n$, as can be seen from the graph $G = mK_{d+1}$ for $d \ge k$ with $n = m(d+1)$. So there is still an asymptotic multiplicative gap of a factor of $2\frac{k+1}{k+2}$. It is worth to mention that there is no known modification of the charming probabilistic proof of the lower bound of Caro-Wei theorem to the situation of $k$-independence that gives a better bound than the Caro-Tuza lower bound. Here, for the sake of being self-contained and to use the same notation, we restate and give the short proof of the Caro-Tuza theorem for graphs. Then we show how to improve this result using further ideas and, in particular, we close the multiplicative gap proving, as a corollary of our main result, that $\alpha_k(G) \ge \frac{k+1}{\lceil d(G) \rceil +k +1} n$. Doing so, we solve of a "folklore" conjecture stated explicitly in \cite{BCHN}.

All along this paper, we will use the following notation and definitions. Let $G$ be a graph. By $V(G)$ we denote the set of vertices of $G$ and $n(G) = |V(G)|$ is the order of $G$. $E(G)$ stands for the set of edges of $G$ and $e(G)$ denotes its cardinality. For a vertex $v \in V(G)$, $\deg(v) = \deg_G(v)$ is the degree of $v$ in $G$. By $\Delta(G)$ we denote the maximum degree of $G$ and by $d(G)$ the average degree $\frac{1}{n(G)} \sum_{v \in V(G)} \deg(v)$. For a subset $S \subseteq V(G)$, we write $G[S]$ for the graph induced by $S$ in $G$ and $\deg_S(v)$ stands for the degree $\deg_{G[S]}(v)$ of $v$ in $G[S]$. Lastly, for a vertex $v \in V(G)$, $G-v$ represents the graph $G$ without vertex $v$ and all the edges incident to $v$.

The paper is divided into five sections. After this introduction section, we deal in section 2 with a first naive approach to obtain a lower bound on $\alpha_k(G)$ by deleting iteratively vertices of maximum degree until certain point where an old theorem of Lov\'asz \cite{Lov} is applied. In section 3, we proceed the same way, taking however a better control on the number of vertices that are deleted and we prove that, for a graph $G$ on $n$ vertices and average degree $d$, $\alpha_k(G) \ge \frac{k+1}{\lceil d \rceil + k + 1} n$, improving the hitherto best general lower bound due to Caro and Tuza. For this purpose, we define a parameter $f(k,d)$ which approaches from below the best possible ratio $\frac{\alpha(G)}{n(G)}$ for graphs $G$ with $d(G) \le d$,  we calculate the exact value of $f(1,d)$ and prove some lower bounds on $f(k,d)$. In Section 4, we develop some upper bounds on $f(k,d)$. Finally, we present in Section 5 some open problems for further research.


\section{The naive approach: first improvement}

Let $f: [0,\infty) \rightarrow \mathbb{R}$ be the function defined by
$$
f_k(x) = \begin{cases}
1-\frac{x}{2(k+1)}, & \text{if } 0 \le x \le k+1\\
\frac{k+2}{2(x+1)}, & \text{if } x \ge k+1.
\end{cases}
$$

Observe the following properties of $f_k(x)$:

\begin{enumerate}
\item[(P1)] $f_k(x)$ is a convex function and is strictly monotone decreasing on $[0, \infty)$.
\item[(P2)] $f_k(i) - f_k(i+1) \ge f_k(j) - f_k(j+1)$, for $j \ge i \ge 0$.
\item[(P3)] $i f_k(i-1) = (i+1) f_k(i)$, for $i \ge k+1$.
\item[(P4)] $f_k(0) = 1$ and $f_k(k+1) = \frac{1}{2}$.
\end{enumerate}

\begin{thm}[Caro-Tuza for Graphs, \cite{CaTu}] \label{Caro-Tuza}
Let $G$ be a graph with degree sequence $d_1, \ldots ,d_n$. Then $\alpha_k(G) \ge \sum_{i = 1}^n  f_k(d_ i)$.
\end{thm}

\begin{proof}
For a subset $X \subseteq V(G)$, define $s(X) = \sum_{x \in X} f_k(\deg_X(x))$. Among all subsets of $X \subseteq V(G)$ such  that $s(X)$ is maximum, choose $B$ such that $B$ has the smallest cardinality. In particular,  $|B| \ge s(B) \ge s(V(G)) = \sum_{x \in V(G)} f_k(\deg(x))$. We will show that $B$ is a $k$-independent set of G. Suppose there is a vertex $y \in B$ such that $\deg_B(y) = d \ge k +1$. Let $y$ be the vertex of maximum degree in $G[B]$. We will show that $s(B \setminus \{y\}) \ge s(B)$, a contradiction to the minimality of $|B|$. For $x \in B\setminus\{y\}$, let $z(x) = 1$ if $xy$ is an edge in $G$ and $0$ otherwise. Then
\begin{eqnarray*}
s(B \setminus \{y\}) &=& \sum_{x \in B\setminus \{y\}} f_k(\deg_{B\setminus \{y\}}(x)) =  \sum_{x \in B\setminus \{y\}} f_k(\deg_B(x) - z(x))\\
                                 &=&  \left( \sum_{x \in B} f_k(\deg_B(x) - z(x)) \right) - f_k(d)\\
                                 &=& s(B) - f_k(d) +  \sum_{x \in B} (f_k(\deg_B(x) - z(x)) - f_k(\deg_B(x)))\\
                                 &=& s(B) - f_k(d) + \sum_{x \in B} z(x) \left(f_k(\deg_B(x)-1) - f_k(\deg_B(x)) \right)\\
                                 &=& s(B) - f_k(d) + \sum_{x \in B \cap N(y)} \left(f_k(\deg_B(x) -1) - f_k(\deg(x)) \right).
\end{eqnarray*}
With (P2) we obtain that the last term is at least $s(B) - f_k(d) + d(f_k(d-1) - f_k(d)) = s(B) - (d+1) f_k(d) + d f_k(d-1)$ and, since $d f_k(d-1) = (d+1) f_k(d)$ by (P3), this is equal to $s(B)$. It follows that $s(B \setminus \{y\}) \ge s(B)$, which is a contradiction to the choice of $B$. Hence, $B$ is a $k$-independent set and thus
$$\alpha_k(G) \ge |B| \ge s(B) \ge s(V) = \sum_{x \in V(G)} f_k(\deg(x)).$$
\end{proof}

Note that, for $k = 0$, Theorem \ref{Caro-Tuza} yields the Caro-Wei bound. By convexity, the above bound yields also the following corollary.

\begin{cor}\label{coro_CT}
For a graph $G$ on $n$ vertices, $\alpha_k(G) \ge f_k(d(G)) n$.
\end{cor}

Note that, for $k=0$, Corollary \ref{coro_CT} yields the Tur\'an bound $\alpha(G) \ge \frac{1}{d(G)+1} n$. Also, if $d(G) \ge k+1$, we obtain from this corollary the following one.

\begin{cor}
Let $G$ be a graph on $n$ vertices. If $d(G) \ge k+1$, then $\alpha_k(G) \ge \frac{k+2}{2(d(G)+1)}n$.
\end{cor}

For a graph $G$, we will denote with $\chi_k(G)$ the \emph{$k$-chromatic number} of $G$, i.e. the minimum number $t$ such that there is a partition $V(G) = V_1 \cup V_2 \cup \ldots V_t$ of the vertex set $V(G)$ such that $\Delta(G[V_i]) \le k$ for all $1 \le i \le t$. The following theorem is due to Lov\'asz.

\begin{thm}[Lov\'asz \cite{Lov}, 1966]
Let $G$ be a graph with maximum degree $\Delta$. If $k_1, k_2, \ldots,$ $k_t \ge 0$ are integers such that $\Delta + 1 = \sum_{i=1}^t (k_i+1)$, then there is a partition $V(G) = V_1 \cup V_2 \cup \ldots \cup V_t$ of the vertex set of $G$ such that $\Delta(G[V_i]) \le k_i$ for $1 \le i \le t$.
\end{thm}

Several proofs and generalizations of Lov\'asz's theorem are known. We refer the reader to \cite{BoMa, BoKoTo, Cat1, Cat2, Rab}. An algorithmic analysis of Lov\'asz theorem with running time $O(n^3)$ is given in \cite{HaLa}. An immediate and well known corollary of Lov\'asz's theorem is Corollary \ref{cor_Lov}, which is useful in the study of defective colorings also known as improper colorings (see \cite{BHHL, CCW, FrHe, HKS}).

\begin{cor}\label{cor_Lov}
If $G$ is a graph of maximum degree $\Delta$, then $\chi_k(G) \le \lceil \frac{\Delta+1}{k+1} \rceil$.
\end{cor}

Since $\alpha_k(G) \ge \frac{n}{\chi_k(G)}$, the following bound proved in 1986 by Hopkins and Staton follows trivially from the above corollary.

\begin{thm}[Hopkins, Staton \cite{HS} 1986] \label{thm_HS}
Let $G$ be a graph of order $n$ and maximum degree $\Delta$.Then
$$\alpha_k(G)\ge\frac{n}{\left\lceil\frac{\Delta+1}{k+1}\right\rceil}.$$
\end{thm}

The following theorem is a direct consequence of Theorem \ref{thm_HS} which generalizes and improves several results concerning relations between $\alpha_p(G)$ and $\alpha_q(G)$ (see e.g. \cite{BCFM}).

\begin{thm}
Let $G$ be a graph and $q \ge p \ge 0$ two integers. Then $\alpha_q(G) \le \left\lceil \frac{q+1}{p+1} \right\rceil \alpha_p(G)$.
\end{thm}

\begin{proof}
Let $S$ be a maximum $q$-independent set of $G$. Then $\Delta(G[S]) \le q$ and, by Theorem \ref{thm_HS},
$$\alpha_p(G) \ge \alpha_p(G[S]) \ge \frac{|S|}{\left\lceil \frac{\Delta(G[S])+1}{p+1} \right\rceil} \ge \frac{\alpha_q(G)}{\left\lceil \frac{q+1}{p+1} \right\rceil},$$
which implies the statement.
\end{proof}

Completing $\Delta + 1$ to the next multiple of $k+1$, the following observation is straightforward from Theorem \ref{thm_HS}.

\begin{obs}\label{obs_HS}
Let $G$ be a graph of order $n$ and maximum degree $\Delta$ and let $r$ be an integer such that $0 \le r \le k$ and $\Delta+1+r\equiv 0 \modulo{k+1}$. Then
$$\alpha_k(G)\ge \frac{k+1}{\Delta+r+1}n.$$
\end{obs}

\begin{proof}
As clearly $\lceil\frac{\Delta+1}{k+1}\rceil = \frac{\Delta+r+1}{k+1}$, Theorem \ref{thm_HS} implies then $\alpha_k(G)\ge \frac{k+1}{\Delta+r+1}n.$
\end{proof}

When the graph is $d$-regular, we can set $\Delta = d = d(G)$ in Observation \ref{obs_HS} and we obtain the following one.

\begin{obs}
Let $G$ be a $d$-regular graph on $n$ vertices and let $r$ be an integer such that $0 \le r \le k$ and $d+1+r\equiv 0 \modulo{k+1}$. Then $\alpha_k(G)\ge \frac{k+1}{d+r+1}n.$
\end{obs}

So this observation shows that indeed, for $d$-regular graphs, we can close the multiplicative gap of $2\frac{k+1}{k+2}$ using Lov\'asz's theorem. This serves as an inspiration to trying to close the multiplicative gap in general.

Note that, in practice, the Hopkins-Staton bound can be poor if the maximum degree is far from the average degree.
So, our first naive strategy will be to delete a vertex with large degree and, if possible, use induction on the number of vertices. Otherwise, if $\Delta(G)$ is near to the average degree $d(G)$, we will apply Theorem \ref{thm_HS}. This is precisely what is done in the next result.

\begin{thm}\label{1st_approach}
Let $G$ be a graph on $n$ vertices. Then $\alpha_k(G) >  \frac{k+1}{d(G) +2k +2}n$.
\end{thm}

\begin{proof}
We will proceed by induction on $n$. If $n=1$, the statement is trivial. If $n=2$, $G$ is either $K_2$ or $\ov{K_2}$. If $G=K_2$, then $d(G) = 1$ and $\frac{k+1}{d(G) +2k +2}n = \frac{2(k+1)}{3+2k} < 1 \le \alpha_k(G)$ for any $k \ge 0$. If $G = \ov{K_2}$, then $d(G) = 0$ and thus $\frac{k+1}{d(G) +2k +2}n = 1 < 2 = \alpha_k(G)$ for all $k \ge 0$. Suppose now that $n \ge 3$ and that the statement holds for $n-1$. Let $G$ be a graph on $n$ vertices and $v \in V(G)$ a vertex of maximum degree $\Delta$. Define $G^* = G - v$. Since any $k$-independent set of $G^*$ is also a $k$-independent set of $G$, $\alpha_k(G) \ge \alpha_k(G^*)$. We distinguish two cases.\\
\emph{Case 1. Suppose that $\Delta \le \lceil d(G) \rceil +k$.} Then, by Observation \ref{obs_HS}, we have
$$\alpha_k(G) \ge \frac{k+1}{\Delta+k+1} n \ge \frac{k+1}{\lceil d(G) \rceil + 2k +1}n > \frac{k+1}{d(G) + 2k + 2}n$$
and we are done.\\
\emph{Case 2. Suppose that $\Delta \ge \lceil d(G) \rceil +k+1$.} By induction and with $\Delta \ge \lceil d(G) \rceil + k+1 \ge d(G)+k+1$, we obtain 
\begin{eqnarray*}
\alpha_k(G)\; \ge \; \alpha_k(G^*) &>& \frac{(k+1)(n-1)}{d(G^*) + 2k+2}  = \frac{(k+1)(n-1)}{\frac{2 e(G^*)}{n-1}+2k+2} \\
                       &=& \frac{(k+1)(n-1)}{\frac{2 e(G) - 2 \Delta}{n-1}+2k+2} = \frac{(k+1)(n-1)}{\frac{d(G) n - 2 \Delta}{n-1}+2k+2}\\
&\ge&  \frac{(k+1)(n-1)}{\frac{d(G) n - 2 (d(G)+k+1)}{n-1}+2k+2} 
                       = \frac{(k+1)n}{(d(G)+2 k+2) \frac{(n-2)n}{(n-1)^2}}\\
                       & >& \frac{k+1}{ d(G)+2k+2}n
\end{eqnarray*}
and the statement follows.
\end{proof}

Note that the bound in previous theorem is better than the Caro-Tuza bound for $k=1$ and $d \ge 8$ and for $k \ge 2$ and $d \ge 2k+5$. Note also that Theorem \ref{1st_approach} already closes the multiplicative factor of $2\frac{k+1}{k+2}$ for fixed $k$ as $d(G)$ grows. However, to obtain an even better lower bound, we need to get more control on the number of vertices of large degrees that are deleted and to apply Observation \ref{obs_HS} in its full accuracy. This will be done in the next section. 

We close this section with the following algorithm for obtaining a $k$-independent set of cardinality at least $\frac{k+1}{d(G) +2k +2}n$ for any graph $G$ on $n$ vertices that yields us the proof of Theorem \ref{1st_approach}.

\noindent
{\bf Algorithm 1}\\[1ex]
INPUT: a graph $G$ on $n$ vertices and $m$ edges.\\[-5ex]
\begin{itemize}
\item[(1)] Compute $\Delta(G)$ and $d(G)$. GO TO (2).
\item[(2)] If $\Delta(G) \le \lceil d(G) \rceil + k$, perform a Lov\'asz partition into $k$-independent sets, choose the largest class $S$ and END. Otherwise choose a vertex $v$ be of maximum degree $\Delta(G)$, set $G:= G - v$ and GO TO (1).\\[-5ex]
\end{itemize}  
OUTPUT: $S$

The algorithm terminates as, at some step, $\Delta(G) \le \lceil d(G) \rceil + k$ must hold (the latest when $G$ is the empty graph). As already mentioned, the Lov\'asz partition requires a running time of $O(n^3)$, while each other step takes at most $O(n)$ time and the number of iteration steps before performing Lov\'asz partition is at most $n$. Hence, the algorithm runs in at most $O(n^3)$ time.


\section{Deletions, partitions and a better lower bound on $\alpha_k(G)$ - second improvement}

\begin{defi}
Let $d, k \ge 0$ be two integers. We define $$f(k,d) = \inf \left\{ \frac{\alpha_k(G)}{n(G)} : G \mbox{ is a graph with } d(G) \le d\right\}.$$
\end{defi}

\begin{obs}
Let $d, k \ge 0$ be two integers. For every graph $G$ on $n$ vertices and average degree $d(G) \le d$, $\alpha_k(G) \ge f(k,d) n$.
\end{obs}

The next theorem shows that $f(k,d)$ is convex as a function of $d$. 

\begin{thm}
Let $d, k, t \ge 0$ be integers and $t \le d$. Then $2f(k,d) \le f(k,d-t) + f(k,d+t)$.
\end{thm}

\begin{proof}
We will show that for any two graphs $G_1$ and $G_2$ such that $d(G_1) \le d-t$ and $d(G_2) \le d+t$, there is a graph $G$ with $d(G) \le d$ such that $2 \frac{\alpha_k(G)}{n(G)} \le \frac{\alpha_k(G_1)}{n(G_1)} + \frac{\alpha_k(G_2)}{n(G_2)}$. Let $G_1$ and $G_2$ be such graphs and let $n(G_i) = n_i$ and $V(G_i) = V_i$, $i = 1,2$. Define the graph $G = n_2 G_1 \cup n_1 G_2$. Then
\begin{eqnarray*}
2n_1n_2 d(G) = n(G) d(G) &=& n_2 \sum_{x \in V_1} \deg_{G_1}(x) + n_1 \sum_{y \in V_2} \deg_{H_2}(y)\\
                                                &=& n_2 n_1 d(H_1) + n_1 n_2 d(G_2) \\
                                                &\le& n_2 n_1 (d-t) + n_1 n_2 (d+t) = 2 n_1 n_2 d,
\end{eqnarray*}
implying that $d(G) \le d$ and thus $f(k,d) \le \frac{\alpha_k(G)}{n(G)}$. Moreover,
$$
2f(k,d) \le 2 \frac{\alpha_k(G)}{n(G)} = \frac{n_2 \alpha_k(G_1) + n_1 \alpha_k(G_2)}{n_1n_2}
= \frac{\alpha_k(G_1)}{n_1} + \frac{\alpha_k(G_2)}{n_1}.
$$
As $G_1$ and $G_2$ were arbitrarily chosen, it follows that $2f(k,d) \le f(k,d-t) + f(k,d+t)$.
\end{proof}

Before coming to the main theorems of this section, we need the following lemmas.

\begin{lem}\label{lem_subg}
Let $d, t \ge 0$ be two integers and let $G$ be a graph on $n$ vertices with average degree $d(G) \le d$. Then $G$ has a subgraph $H$ such that either  $\Delta(H) \le d+t-1$ and $n(H) \ge n - \lfloor \frac{n}{d+2t+1}\rfloor$ or $d(H) \le d-1$ and $n(H) = n - \lceil \frac{n}{d+2t+1} \rceil$.
\end{lem}

\begin{proof}
For an $r \ge0$, let $\{v_1, v_2, \ldots, v_r\}$ be a set of vertices of maximum cardinality such that $\deg_{G_{i+1}}(v_i) \ge d+t$, where $G_{i+1} = G_i - v_i$ and $G_0 = G$. Suppose first that $r \le \lfloor \frac{n}{d+2t+1} \rfloor$ and let $H = G_{r+1}$. Then $H$ has at least $n - r \ge n - \lfloor \frac{n}{d+2t+1} \rfloor$ vertices and $\Delta(H) \le d+t-1$. Now suppose that $r \ge \lceil \frac{n}{d+2t+1} \rceil$. Let now $H = G_{q+1}$, where $q = \lceil \frac{n}{d+2t+1} \rceil$. Then $n(H) = n - \lceil \frac{n}{d+2t+1} \rceil$. Further,
\begin{eqnarray*}
d(H) = \frac{2 e(H)}{n(H)} \le \frac{2(e(G) - (d+t)q)}{n-q} = \frac{dn - 2(d+t)q}{n-q} = \frac{d(n-\frac{2(d+t)}{d}q)}{n-q}.
\end{eqnarray*}
Since, for any real numbers $a \ge 0$ and $b \ge 1$, the function $\frac{a - bx}{a-x}$ is monotonically decreasing in $[0, \infty)$, setting $a = n$ and $b = \frac{2(d+t)}{d}$, we obtain with $q = \lceil \frac{n}{d+2t+1} \rceil \ge \frac{n}{d+2t+1}$
\begin{eqnarray*}
d(H) &\le& \frac{d(n-\frac{2(d+t)}{d}q)}{n-q} \le \frac{d(n- \frac{2(d+t)}{d} \frac{n}{d+2t+1})}{n- \frac{n}{d+2t+1}}\\
&=& \frac{d(d+2t+1) - 2(d+t)}{d+2t} = \frac{d(d+2t) - (d+2t)}{d+2t} = d-1.
\end{eqnarray*}
Hence, we have shown that $G$ has a subgraph $H$ with either $d(H) \le d-1$ and $n(H) = n - \lceil \frac{n}{d+2t+1} \rceil$ or $\Delta(H) \le d+t-1$ and $n(H) = n - \lfloor \frac{n}{d+2t+1}\rfloor$.
\end{proof}

The following corollary is straightforward from this lemma.

\begin{cor}\label{cor_subg}
Let $d, t \ge 0$ be two integers. Let $G$ be a graph on $n$ vertices with average degree $d(G) \le d$ and such that $d+2t+1$ divides $n$. Then $G$ has a subgraph $H$ on $n(H) \ge \frac{d+2t}{d+2t+1}n$ vertices such that either $d(H) \le d-1$ or $\Delta(H) \le d+t-1$.
\end{cor}


\begin{lem}\label{div}
Let $G$ be a graph on $n$ vertices with average degree $d(G) \le d$ and such that $d+2t+1$ does not divide $n$. Then there is a graph $H$ such that $d+2t+1$ divides $m = n(H)$, $d(H) = d(G) \le d$ and $\frac{\alpha_k(H)}{m} = \frac{\alpha_k(G)}{n}$.
\end{lem}

\begin{proof}
Let $H = (d+2t+1) G$ be the graph consisting of $d+2t+1$ copies of $G$. Then $m = n(H) = (d+2t+1)n$ is multiple of $d+2t+1$, $d(H) = d(G)$ and $\frac{\alpha_k(H)}{m} = \frac{(d+2t+1) \alpha_k(G)}{(d+2t+1)n} = \frac{\alpha_k(G)}{n}$.
\end{proof}

Let $n$ be an even integer. We denote with $J_n$ the graph consisting of a complete graph on $n$ vertices minus a $1$-factor. We are now ready to present the exact value of $f(1,d)$ and the consequences of this result. 

\begin{thm}\label{thm_k=1}
Let $d \ge 0$ be an integer. Then the following statements hold.
\begin{itemize}
\item[(1)] 
$f(1,d) = \begin{cases} \frac{2}{d+2}, &\text{if } d \equiv 0 \modulo{2}\\
                                                          \frac{2(d+2)}{(d+1)(d+3)}, &\text{if } d \equiv 1 \modulo{2}.
                                \end{cases}$
\item[(2)] The equality $f(1,d) = \frac{\alpha_1(G)}{n(G)}$ is attained by the graph $G = J_{d+2}$, when $d$ is even, and by $G=(d+3) J_{d+1} \cup (d+1) J_{d+3}$, when $d$ is odd.
\item[(3)] $f(1,d) \ge \frac{2}{d+2}$.
\item[(4)] For every graph $G$ on $n$ vertices, $\alpha_1(G) \ge \frac{2n}{\lceil d(G) \rceil +2}$.
\end{itemize}
\end{thm}

\begin{proof}
(1) We will prove by induction on $d$ that 
$$f(1,d) \ge \begin{cases} \frac{2}{d+2}, &\text{if } d \equiv 0 \modulo{2}\\
                                                          \frac{2(d+2)}{(d+1)(d+3)}, &\text{if } d \equiv 1 \modulo{2}.
                                \end{cases}$$
For $d = 0$, clearly $f(1,0) = 1 = \frac{2}{0+2}$, as the only possible graph $G$ with $d(G) \le 0$ is the empty graph. For $d=1$, let $G$ be a graph with $d(G) \le 1$. Setting $t=1$, we can suppose by Lemma \ref{div} that $4$ divides $n(G) = n$. Hence, Corollary \ref{cor_subg} implies that there is a subgraph $H$ of $G$ on $n(H) \ge \frac{3}{4}n$ vertices with $d(H) \le 0$ or $\Delta(H) \le 1$. In both cases we have clearly $\alpha_1(G) \ge \alpha_1(H) = n(H) \ge \frac{3}{4}n$ and hence $f(1,1) = \inf\{\frac{\alpha_1(G)}{n(G)} : G \mbox{ graph with } d(G) \le 1 \} \ge \frac{3}{4}= \frac{2(1+2)}{(1+1)(1+3)}$. \\
Assume we have proved the statement for $f(1,d-1)$. Now we will prove it for $f(1,d)$, where $d > 1$. Let $G$ be a graph on $n$ vertices such that $d(G) \le d$. We distinguish two cases.\\
\emph{Case 1. Suppose that $d \equiv 0 \modulo{2}$.} Setting $t=0$, we can suppose by Lemma \ref{div} that $d+1$ divides $n$. By Corollary \ref{cor_subg}, there is a subgraph $H$ of $G$ on at least $\frac{d}{d+1}n$ vertices with either $d(H) \le d-1$ or $\Delta(H) \le d-1$. Hence, in both cases $d(H) \le d-1$ and thus, by induction, we have 
$$\alpha_1(G) \ge \alpha_1(H) \ge f(1,d-1) n(H) \ge  \frac{2(d+1)d}{d(d+2)(d+1)} n = \frac{2}{d+2}n.$$
Hence, $f(1,d) =  \inf\{\frac{\alpha_1(G)}{n(G)} : G \mbox{ graph with } d(G) \le d \} \ge \frac{d}{d+2}$ and we are done.\\
\emph{Case 2. Suppose that $d \equiv 1 \modulo{2}$.} Set $t =1$. By Lemma \ref{div}, we can suppose that $d+3$ divides $n$. By Corollary \ref{cor_subg}, there is a subgraph $H$ of $G$ on at least $\frac{d+2}{d+3}n$ vertices with either $d(H) \le d-1$ or $\Delta(H) \le d$. If $d(H) \le d-1$, we can apply the induction hypothesis on $H$ and we obtain
$$\alpha_1(G) \ge \alpha_1(H) \ge f(1,d-1) n(H) \ge  \frac{2(d+2)}{(d+1)(d+3)} n$$
and we are done. Suppose finally that $\Delta(H) \le d$. Then, by Theorem \ref{thm_HS} and as $d$ is odd, we have
$$\alpha_1(G) \ge \alpha_1(H) \ge \frac{n(H)}{\left\lceil \frac{\Delta(H) + 1}{2} \right\rceil} \ge \frac{ \frac{d+2}{d+3}n }{\left\lceil \frac{d + 1}{2}\right\rceil}=  \frac{2(d+2)}{(d+1)(d+3)} n.$$
Thus, in both cases, $f(1,d) = \inf\{\frac{\alpha_1(G)}{n(G)} : G \mbox{ graph with } d(G) \le d \} \ge \frac{2(d+2)}{(d+1)(d+3)}$
Hence, by induction, the statement holds.\\
Let $d$ be even. Clearly $\alpha_1(J_{d+2})= 2$ and hence, $f(1,d) \le \frac{\alpha_1(J_{d+2})}{d+2} = \frac{2}{d+2}$. For $d$ odd, the graph $G = (d+3) J_{d+1} \cup (d+1) J_{d+3}$ has $\alpha_1(G) = (d+3) 2 + (d+1)2 = 4(d+2)$. Hence $f(1,d) \le \frac{\alpha_1(G)}{n(G)} =  \frac{4(d+2)}{2(d+1) (d+3)}$. Together with the inequalities proven above, it follows
$$f(1,d) = \begin{cases} \frac{2}{d+2}, &\text{if } d \equiv 0 \modulo{2}\\
                                                          \frac{2(d+2)}{(d+1)(d+3)}, &\text{if } d \equiv 1 \modulo{2}.
                                \end{cases}$$
(2) This follows from the discussion in (1).\\
(3) It is easily seen that $\frac{2(d+2)}{(d+1)(d+3)} \ge \frac{2}{d+2}$. Hence we have always $f(1,d) \ge \frac{2}{d+2}$.   \\
(4) From item (2), we obtain $\alpha_1(G) \ge f(1, \lceil d(G) \rceil) n \ge \frac{2}{\lceil d(G) \rceil+2}$ n.                    
\end{proof}

We can now state and prove our main result generalizing the proof of Theorem \ref{thm_k=1} to arbitrary $k$ and $d$.

\begin{thm}\label{main}
Let $d, k \ge 0$ be two integers. Then the following statements hold.
\begin{itemize}
\item[(1)] $f(k,d) \ge \frac{(k+1)(d+2t)}{(d+k+t+1)(d+t)} \ge \frac{k+1}{d+k+1}$, where $t$ is such that $d \equiv k+1 - t  \modulo{k+1}$ and $1 \le t \le k+1$.
\item[(2)] For $k \ge d$, $f(k,d) \ge \frac{2k+2-d}{2k+2}$. For $k \ge d = 1$, the bound is realized by the graph $K_{1,k+1} \cup kK_1$ and thus $f(k,1) = \frac{2k+1}{2k+2}$.
\item[(3)] For any graph $G$ on $n$ vertices, $\alpha_k(G) \ge \frac{k+1}{\lceil d(G) \rceil + k+1} n$.
\end{itemize}
\end{thm}

\begin{proof}
(1) We will proceed to prove the inequality $f(k,d) \ge \frac{(k+1)(d+2t)}{(d+k+t+1)(d+t)}$ by induction on $d$. If $d=0$, then $d \equiv (k+1) - (k+1)$ and clearly $f(k,0) = 1 = \frac{(k+1) (0 + 2)(k+1))}{(0+k+(k+1)+1) (0+(k+1))}$, as the only possible graph $G$ with $d(G) \le 0$ is the empty graph.

Assume $f(k,d-1) \ge \frac{(k+1)(d-1+2t')}{(d+k+t')(d-1+t')}$, where $d-1 \equiv k+1 - t' \modulo{k+1}$, $1 \le t' \le k+1$, and $d \ge 1$. We will prove the statement for $d$. Herefor, we distinguish two cases.\\
\emph{Case 1. Suppose that $d \equiv 0 \modulo{k+1}$.} Then $t = k+1$. Let $G$ be a graph on $n$ vertices such that $d(G) \le d$. By Lemma \ref{div} (setting there $t=0$), we can suppose that $d+1$ divides $n$. Then from Corollary \ref{cor_subg} it follows that there is a subgraph $H$ of $G$ on at least $\frac{d}{d+1}n$ vertices such that $d(H) \le d-1$ or $\Delta(H) \le d-1$. In both cases we have $d(H) \le d-1$. Then, as $d-1 \equiv (k+1) - 1 \modulo{k+1}$, we obtain by induction
\begin{eqnarray*}
\alpha_k(G) &\ge& \alpha_k(H) \ge \frac{(k+1)(d+1)}{(d+k+1)d} n(H) \ge \frac{k+1}{d+k+1}n \\ 
&=& \frac{(k+1)(d+2t)}{(d+2t)(d+k+1)}n = \frac{(k+1)(d+2t)}{(d+k+t+1)(d+t)}n.
\end{eqnarray*}
Thus, $f(k,d) =  \inf\{\frac{\alpha_k(G)}{n(G)} : G \mbox{ graph with } d(G) \le d \} \ge \frac{(k+1)(d+2t)}{(d+k+t+1)(d+t)}$ and we are done.\\
\emph{Case 2. Suppose that $d \equiv k+1 - t \modulo{k+1}$ for  some $t$ with $1 \le t \le k$.}
By Corollary \ref{cor_subg}, there is a subgraph $H$ of $G$ on $n(H) \ge \frac{d+2t}{d+2t+1}n$ vertices with either $d(H) \le d-1$ or $\Delta(H) \le d+t-1$. If $\Delta(H) \le d+t-1$, then Theorem \ref{thm_HS} yields
\begin{eqnarray*}
\alpha_k(G) \ge \alpha_k(H) \ge \frac{n(H)}{\left\lceil \frac{\Delta(H)+1}{k+1} \right\rceil} 
                                           \ge \frac{\frac{d+2t}{d+2t+1}n}{\left\lceil \frac{d+t}{k+1} \right\rceil}  
                                           &=& \frac{(k+1)(d+2t)}{(d+2t+1)(d+t)}n\\
                                           &\ge& \frac{(k+1)(d+2t)}{(d+k+t+1)(d+t)} n.
\end{eqnarray*} 
Hence, $f(k,d) =  \inf\{\frac{\alpha_k(G)}{n(G)} : G \mbox{ graph with } d(G) \le d \} \ge \frac{(k+1)(d+2t)}{(d+k+t+1)(d+t)}$                                           
and we are done. Suppose now that $d(H) \le d-1$. Since $d -1 \equiv (k+1) - (t+1)$, we obtain by induction
\begin{eqnarray*}
\alpha_k(G) \ge \alpha_k(H) &\ge& \frac{(k+1)((d-1)+2(t+1))}{((d-1)+k+(t+1)+1)((d-1)+(t+1))} n(H) \\
                                                                    &\ge& \frac{(k+1)(d+2t+1)}{(d+k+t+1)(d+t)}\cdot \frac{d+2t}{d+2t+1} n\\
                                                                    &=& \frac{(k+1)(d+2t)}{(d+k+t+1)(d+t)} n.
\end{eqnarray*} 
Thus, again, $f(k,d) \ge \frac{(k+1)(d+2t)}{(d+k+t+1)(d+k+1)}$ and Case 2 is done. \\
Hence, by induction, the statement holds. Finally, the inequality $ \frac{(k+1)(d+2t)}{(d+k+t+1)(d+t)} \ge \frac{k+1}{d+k+1}$ follows easily.\\
(2) Let $k \ge d$ and let $t$ be such that $d \equiv k+1 - t \modulo{k+1}$ and $1 \le t \le k+1$. Then $d = k+1-t$. Hence, with (1),
$$f(d,k) \ge  \frac{(k+1)(d+2t)}{(d+k+t+1)(d+t)} = \frac{(k+1)(d+ 2(k+1-d))}{(d+k+(k+1-d)+1)(d+(k+1-d))} = \frac{2k+2-d}{2k+2}.$$
Let $G = K_{1,k+1} \cup kK_1$. Then $\alpha_k(G) = 2k+1$, $n(G) = 2k+2$ and $d(G) = \frac{2k+2}{2k+2} = 1$. Hence, $\frac{2k+1}{2k+2} \le f(k,1) \le \frac{\alpha_k(G)}{n(G)} = \frac{2k+1}{2k+2}$, obtaining thus equality.\\
(3) If $G$ is a graph on $n$ vertices, then, using (1), we obtain 
$$\frac{\alpha_k(G)}{n} \ge f(k, \lceil d(G) \rceil) \ge \frac{k+1}{\lceil d(G) \rceil +k+1}.$$
\end{proof}

The proofs of Lemma \ref{lem_subg} and Theorem \ref{main} yield us an algorithm for finding, for any graph $G$ on $n$ vertices, a $k$-independent set of cardinality at least $\frac{k+1}{\lceil d(G) \rceil + k+1} n$. It works the following way. It computes $d=d(G)$ and $\Delta(G)$ and finds the integer $t$ such that $0 \le t \le k$ and $d \equiv k+1 - t  \modulo{k+1}$ (note that the case $t=0$ corresponds here to the case $t=k+1$ of Theorem \ref{main}). Then it checks if the graph satisfies the condition $\Delta(G) \le d+t-1$. If so, then it performs a Lov\'asz partition into $k$-independent sets, selects the largest set from it and gives this as output. If not, then a vertex of maximum degree is deleted and the condition on the maximum degree is checked again on the remaining graph. This deletion step is repeated up to $ \lceil\frac{n}{d+2t+1} \rceil$ times, as, by Lemma \ref{lem_subg}, if the maximum degree is still larger than $d+t-1$, then we are left with a graph with smaller average degree, with which the algorithm starts over again, doing here the inductive step of Theorem \ref{main}. \\

\noindent
{\bf Algorithm 2}\\[1ex]
INPUT: a graph $G$ on $n$ vertices and $m$ edges.\\[-5ex]
\begin{itemize}
\item[(1)] Compute $\Delta(G)$ and $d(G)$.  Set $d = \lceil d(G) \rceil$ and determine $t$ such that $0 \le t \le k$ and $d \equiv k+1 - t  \modulo{k+1}$. Set $r:=0$ and  GO TO (2).
\item[(2)] If $\Delta(G) \le d+t-1$, perform a Lov\'asz partition into $k$-independent sets, choose the largest class $S$ and END. Otherwise GO TO (3).
\item[(3)]  Set $r := r+1$. If $r > \lceil\frac{n}{d+2t+1} \rceil$, set $n := n -  \lceil \frac{n}{d+2t+1} \rceil$ and GO TO (1). Otherwise choose a vertex $v$ of maximum degree $\Delta(G)$, set $G:= G - v$, compute $\Delta(G)$ and GO TO (2). \\[-5ex]
\end{itemize}  
OUTPUT: $S$

The algorithm terminates as, at some step, $\Delta(G) \le \lceil d(G) \rceil +t-1$ must hold (the latest when $G$ is the empty graph).
Again, the algorithm has a running time of at most $O(n^3)$.


\section{Upper bounds on $f(k,d)$ and determination of $f(k,d)$ for further small values}

Observe that after Theorems \ref{thm_k=1}(1) and \ref{main}(2), we know the exact value of $f(k,d)$ in case $\min\{d,k\} \le 1$. The first pair $(k,d)$ for which an exact value of $f(k,d)$ is not known yet is $(2,2)$. In this section, we develop several upper bounds on $f(k,d)$ as a starting point to future research to obtain further exact values of $f(k,d)$. We will use the following Theorem.

\begin{thm}[see \cite{Bol}, p.108]\label{thm_bol}
Let $r, g \ge 3$ be two integers. If $m$ is an integer with $m \ge \frac{(r-1)^{(g-1)}-1}{r-2}$, then there exists an $r$-regular graph of girth at least $g$ and order $2m$.
\end{thm}

Define the function $h(r) = \frac{(r-1)^{r+3}-1}{r-2}$. We will use the particular form of this theorem with $m \ge h(r)$, implying that there is an $r$-regular graph of girth at least $r+4$ and order $2m$.

In the proof of the following theorem, we use the following notation. $\overline{G}$ denotes the complementary graph of $G$. If $F \subseteq E(G)$, then $G-F$ represents the graph $G$ without the edges contained in $F$. For $q \le n$, $K_n - E(K_q)$ stands for the complete graph $K_n$ without the edges of a subgraph $K_q$. Further given two graphs $G$ and $H$, $G \cup H$ is the graph consisting of one copy of $H$ and one copy of $G$ and $qG$ denotes the union of $q$ copies of $G$. Finally, the girth of a graph $G$ is denoted by $g(G)$.

\begin{thm}\label{upper_bounds}
Let $d, k \ge 0$ be two integers. Then the following statements hold.
\begin{itemize}
\item[(1)] For $d \ge k$, $\frac{k+1}{d+k+1} \le f(k,d) \le \frac{k+1}{d+1}$.
\item[(2)] For $d > k$, $d \equiv 0 \modulo{2}$ and $k \equiv 1 \modulo{2}$, $f(k,d) \le \frac{k+1}{d+2}$.
\item[(3)] For $d > k$, $f(k,d) \le \frac{k+2}{d+3}$.
\item[(4)] For $k \ge 3$, $d \ge 2h(k) - k-1$ and $d+k+1 \equiv 0 \modulo{2}$, $f(k,d) \le \frac{k+2}{d+k+1}$.
\item[(5)] For $2 \le d \le 4+6q$, where $q \ge 0$ is an integer, $\frac{3}{d+3} \le f(2,d) \le \frac{3}{d+1+\frac{1}{q+1}}$.
\item[(6)] For $k \ge 2$, $\frac{k}{k+1} \le f(k,2) \le \frac{k+1}{k+2+\frac{1}{k+1}}$.
\item[(7)] For $k \ge 3$, there is a constant $c > 0$ auch that $f(k,d) < \frac{k+2}{d+ c(\frac{d}{2})^{\frac{1}{k+2}} +1}$.
\end{itemize}
\end{thm}

\begin{proof}
(1) The lower bound follows from Theorem \ref{main}. The upper bound follows from
$f(k,d) \le \frac{\alpha_k(K_{d+1})}{d+1} = \frac{k+1}{d+1}$.\\
(2) Let $G$ be the graph $K_{d+2}$ minus a $1$-factor (this is possible, as $d$ is assumed even). Then $d(G) = d$. Let $T\subseteq V(G)$ be any subset of $k+2$ vertices. As $k+1 \equiv 0 \modulo{2}$, not every vertex of $T$ is covered by the edges of the $1$-factor in $\overline{G}[T]$. Hence, at least one vertex from $T$ is adjacent in $G$ to all other vertices from $T$. Hence, no subset of $k+2$ vertices can be a $k$-independent set and thus $\alpha_k(G) \le k+1$. This implies $f(k,d) \le \frac{k+1}{d+2}$.\\
(3) Let $d >k$. Consider the graph $G = K_{d+3} - E(C_{d+3})$, where $C_{d+3}$ is a cycle of length $d+3$ in $K_{d+3}$. Then $d(G) =d$. Let $T \subseteq V(G)$ a subset of $k+3$ vertices. Since $k+3 < d+3 = n(G)$, the graph $\overline{G}[T]$ contains no cycles. Hence it there is at least one vertex in $v \in V(T)$ which is adjacent in $G[T]$ to all but at most one vertex and hence $\deg_{G[T]}(v) \ge k+1$. This implies that $\alpha_k(G) \le k+2$ and thus $f(k,d) \le \frac{k+2}{d+3}$.\\
(4) Let $k \ge 3$, $d \ge 2h(k) - k-1$ and $d+k+1 \equiv 0 \modulo{2}$. By Theorem \ref{thm_bol}, there is a $k$-regular graph $H$ with $g(H) \ge k+4$ and $n(H) = d+k+1 = n$. Consider now the graph $G = K_n - E(H)$. Then $d(G) = n-1-k = d$. Let $T \subseteq V(G)$ be a subset of $k+3$ vertices. Since $g(H) \ge k+4$, $\overline{G}[T]$ is a forest. Hence it there is at least one vertex in $v \in V(T)$ which is adjacent in $G[T]$ to all but at most one vertex and hence $\deg_{G[T]}(v) \ge k+1$. Thus, $\alpha_k(G) \le k+2$ and we obtain $f(k,d) \le \frac{k+2}{d+k+1}$.\\
(5) Consider the graph $G = (K_{d+2}-E(K_3)) \cup q(K_{d+1}-E(K_3))$. Then $n(G) = (q+1) d + q+2$ and $d(G) n(G) = (d-1)(d+1)+3(d-1) + q((d-2)d + 3(d-2)) = (q+1)d^2+(q+3)d-(4+6q)$. Since $d \le 4+6q$, it follows that 
$$d(G) = \frac{(q+1)d^2+(q+3)d-(4+6q)}{(q+1)d+q+2} \le \frac{(q+1)d^2+(q+3)d-d}{(q+1)d+q+2} = d.$$ 
As clearly $\alpha_2(G) = 3(q+1)$, we obtain therefore, together with Theorem \ref{main} (1), 
$$\frac{3}{d+3} \le f(2,d) \le \frac{3(q+1)}{(q+1)d+q+2} = \frac{3}{d+\frac{q+2}{q+1}} = \frac{3}{d+1+\frac{1}{q+1}}.$$ 
(6) Let $k \ge 2$ and consider the graph $G = (K_{k+3}-E(K_{k+1})) \cup k K_{1,k+1}$. Then $n(G) = k+3 + k(k+2) = k^2+3k+3$ and $d(G) n(G) = 2(k+2) + (k+1) 2 + 2k(k+1) = 2(k^2+3k+3) = 2 n(G)$. Hence, $d(G) = 2$. Moreover, it is easy to see that $\alpha_k(G)= (k+1)^2$. Thus this implies that
$$f(k,2) \le \frac{\alpha_k(G)}{n(G)} = \frac{(k+1)^2}{k^2+3k+3} = \frac{k+1}{k+2+\frac{1}{k+1}}.$$
Together with the bound from item(2) of Theorem \ref{main}, we obtain
$$\frac{k}{k+1} \le f(k,2) \le \frac{k+1}{k+2+\frac{1}{k+1}}.$$
(7) By Theorem \ref{thm_bol} there is an $r$-regular graph $H$ with $g(H) \ge k+4$ and $n = n(H) \ge \frac{2((r-1)^{k+3}-1)}{r-2}$. Take $n$ even and let $G = K_n - E(H)$. Then $d = d(G) = n-1-r$. Let $T \subseteq V(G)$ be a subset of $k+3$ vertices. As $g(H) \ge k+4$, $\overline{G}[T]$ is a forest and thus there is a vertex in $T$ which is adjacent in $G$ to all other vertices from $T$ with the exception of at most one. Hence, $T$ cannot be a $k$-independent set and thus $\alpha_k(G) \le k+2$. This implies that $f(k,d) \le \frac{\alpha_k(G)}{n} \le \frac{k+2}{d+r+1}$. 
As $d \sim 2r^{k+2}$, we have $r \sim (\frac{d}{2})^{\frac{1}{k+2}}$, and thus there is a constant $c > 0$ such that $r = c (\frac{d}{2})^{\frac{1}{k+2}}$, implying that $f(k,d) \le  \frac{k+2}{d+c(\frac{d}{2})^{\frac{1}{k+2}}+1}$.
\end{proof}


\section{Open problems}

We close this paper with the following open problems.

\begin{prob}
Is $f(k,d)$ in fact a minimum for every $k$ and $d$? Namely, does
$$\inf \left\{\frac{\alpha_k(G)}{n(G)} : G \mbox{ graph with } d(G) \le d \right\} = \min \left\{\frac{\alpha_k(G)}{n(G)} : G \mbox{ graph with } d(G) \le d\right\}$$
hold?
\end{prob}

In case the answer to this problem is positive, this may have several consequences in computing $f(k,d)$.

\begin{prob}
Is the bound $f(k,d) \ge \frac{2k+2-d}{2k+2}$ of Theorem \ref{main} (2) sharp for $k \ge d \ge 2$?
\end{prob}

Below are the best possible bounds on $f(2,d)$ we have for $k = 0,1, \ldots, 10$.

\begin{center}
    \begin{tabular}{|c|c|c|c|c|}
        \hline
                &              lower          &        upper          &                                            & theorem used\\
        $d$ & bound$^*$  &   bound &  graph for upper bound &  for\\ 
                &&                     &                                   & upper bound \\ \hline 
        $0$ & $1$ & $1$ & $K_1$& -\\ \hline
        $1$ & $5/6$ & $5/6$ & $K_{1,k+1} \cup kK_1$& \ref{main}(2)\\ \hline
        $2$ & $2/3$ & $9/13$ & $(K_5 -E(K_3)) \cup 2K_{1,3}$& \ref{upper_bounds} (6)\\ \hline
        $3$ & $1/2$ & $3/5$ & $K_{5} - E(K_3)$& \ref{upper_bounds} (5), $q=0$\\ \hline
        $4$ & $4/9$ & $1/2$ & $K_{6} - E(K_3)$& \ref{upper_bounds} (5), $q=0$\\ \hline
        $5$ & $7/18$ & $6/13$ & $(K_{7} - E(K_3)) \cup (K_{6} - E(K_3))$& \ref{upper_bounds} (5), $q=1$\\ \hline
        $6$ & $1/3$ & $2/5$ & $(K_{8} - E(K_3)) \cup (K_{7} - E(K_3))$&  \ref{upper_bounds} (5), $q=1$\\ \hline
        $7$ & $11/36$ & $6/17$ & $(K_{9} - E(K_3)) \cup (K_{8} - E(K_3))$&  \ref{upper_bounds} (5), $q=1$\\ \hline
        $8$ & $5/13$ & $6/19$ & $(K_{10} - E(K_3)) \cup (K_{9} - E(K_3))$&  \ref{upper_bounds} (5), $q=1$\\ \hline
        $9$ & $1/4$ & $2/7$ & $(K_{11} - E(K_3)) \cup (K_{10} - E(K_3))$&  \ref{upper_bounds} (5), $q=1$\\ \hline
        $10$ & $7/30$ & $6/23$ & $(K_{12} - E(K_3)) \cup (K_{11} - E(K_3))$&  \ref{upper_bounds} (5), $q=1$\\ \hline
    \end{tabular}
    \end{center}
    \mbox{}\\[-3ex]
    \mbox{}
    \hspace{8ex}\small{*Lower bounds are from Thm. \ref{main}(2) in case $d=1$ and Thm. \ref{main}(1) else.}

\mbox{}
\begin{prob}
Improve upon the values given in the table.
\end{prob}

In order to better understand $f(k,d)$, we can define 
$$f(k,d,\Delta) = \inf \left\{\frac{\alpha_k(G)}{n(G)} : G \mbox{ is a graph with } d(G) \le d \mbox{ and } \Delta(G) \le \Delta\right\},$$
where $\Delta \ge d$, and $k$ are all nonnegative integers. Observe that $f(k,d) = \inf \{f(k,d,\Delta) : \Delta \ge d\}$ and hence a knowledge on $f(k,d,\Delta)$ may help in obtaining better bounds on $f(k,d)$. For instance, let us take $f(2,2,3)$. Observe that, from Theorem \ref{main} (2), $f(2,2,3) \ge f(2,2) \ge \frac{2}{3}$. Further, consider the graph $G = R_8 \cup 4K_{1,3}$ on $24$ vertices, where $R_8$ is the graph depicted below (note that $R_8$ is the extremal graph for Reed's upper bound of $\frac{3}{8}n$ on the domination number for graphs on $n$ vertices with minimum degree at least $3$), and observe that $\alpha_2(G) = 17$, $n(G) = 24$ and $\Delta(G) = 3$. Then, it follows that $\frac{2}{3} \le f(2,2) \le f(2,2,3) \le \frac{17}{24}$. 

\begin{figure}[h]
\begin{center}
\psset{unit=0.7cm, linewidth=0.03cm}
   \begin{pspicture}(-5,-1)(8,3.5)
   \cnode*(1,0){0.15}{1}
   \cnode*(2,0){0.15}{2}
   \cnode*(3,1){0.15}{3}
   \cnode*(3,2){0.15}{4}
   \cnode*(2,3){0.15}{5}
   \cnode*(1,3){0.15}{6}
   \cnode*(0,2){0.15}{7}
   \cnode*(0,1){0.15}{8} 
   \ncline{1}{2}
   \ncline{2}{3}
   \ncline{3}{4}
   \ncline{4}{5}
   \ncline{5}{6}
   \ncline{6}{7}
   \ncline{7}{8}
   \ncline{8}{1}
   \ncline{1}{5}
   \ncline{2}{6}
   \ncline{3}{7}
   \ncline{4}{8}
   \put(0,-1){\small{The graph $R_8$.}}
   \end{pspicture}
   \end{center}
   \end{figure}

But if we consider for instance the graph $H= (K_5 -E(K_3)) \cup 2K_{1,3}$, then we have there $\alpha_2(H) = 9$, $n(H) = 13$ and $\Delta(H) = 4$ and thus $\frac{2}{3} \le f(2,2) \le f(2,2,4) \le \frac{9}{13}$, which is better than the bound $\frac{17}{24}$ obtained with the graph $G$. Thus, we would like to state the following question.

\begin{prob}
Obtain lower and upper bounds on $f(k,d, \Delta)$.
\end{prob}

\end{document}